\documentclass[12pt,reqno]{amsart}

\usepackage[T1]{fontenc}
\usepackage[utf8]{inputenc}
\usepackage{fourier} 
\usepackage{mathrsfs}
\usepackage{amsmath, amsfonts, amsthm, amssymb, amsxtra, bbm}
\usepackage{dsfont} 
\usepackage{mathtools}
\usepackage{physics}
\usepackage{esint}

\usepackage{graphicx}
\usepackage{color, xcolor}
\usepackage{pdflscape}
\usepackage{listings}

\usepackage[colorlinks=true]{hyperref}
\hypersetup{
  urlcolor=blue,
  linkcolor=blue,
  citecolor=red,
  anchorcolor=blue
}

\newtheorem{thm}{Theorem}

\newtheorem{lem}[thm]{Lemma}
\newtheorem{corollary}[thm]{Corollary}

\newcommand{\R}{{\mathbb R}}

\newcommand{\Z}{{\mathbb Z}}
\newcommand{\C}{{\mathbb C}}
\renewcommand{\S}{{\mathbb S}}
\newcommand{\Sph}{{\mathbb S}}

\newcommand{\be}[1]{\begin{equation}\label{#1}}
\newcommand{\ee}{\end{equation}}

\renewcommand{\(}{\left(}
\renewcommand{\)}{\right)}

\newcommand{\iR}[1]{\int_{\R}{#1}\,ds}

\newcommand{\iL}[2]{\int_{\R}{#1}\,d\mu_{#2}}
\newcommand{\isn}[2]{\int_{-1}^{+1}{#1}\(1-z^2\)^{#2}\,dz}

\newcommand{\irdsph}[3]{\int_0^\infty\kern-5pt\int_{\S^{d-1}}#2\,{#1}^{\kern1pt #3}\,\frac{d{#1}}{{#1}}}

\DeclareMathOperator{\re}{Re}

\definecolor{darkgreen}{rgb}{.0,.4,.2}

\usepackage{etoolbox}
\newcounter{taggedeq}
\pretocmd{\equation}{\stepcounter{taggedeq}}{}{}

\makeatletter
\@namedef{subjclassname@2020}{\textup{2020} Mathematics Subject Classification}
\makeatother
\newcommand{\msc}[1]{\href{https://zbmath.org/classification/?q=cc:#1}{#1}}

\title[Symmetry in CKN inequalities for two-dimensional spinors]{Symmetry and symmetry breaking in interpolation inequalities for two-dimensional spinors -- Preliminary results}

\author[J.~Dolbeault]{Jean Dolbeault}
\address[Jean Dolbeault]{CEREMADE (CNRS UMR n$^\circ$ 7534), PSL university, Universit\'e Paris-Dauphine, Place de Lattre de Tassigny, 75775 Paris 16, France}
\email{dolbeaul@ceremade.dauphine.fr}

\author[R.L.~Frank]{Rupert L. Frank}
\address[Rupert L. Frank]{Mathe\-matisches Institut, Ludwig-Maximilians Universit\"at M\"unchen, The\-resienstr.~39, 80333 M\"unchen, Germany, and Munich Center for Quantum Science and Technology, Schel\-ling\-str.~4, 80799 M\"unchen, Germany, and Mathematics 253-37, Caltech, Pasa\-de\-na, CA 91125, USA}
\email{r.frank@lmu.de}

\author[J.~Weixler]{Jonte Weixler}
\address[Jonte Weixler]{CEREMADE (CNRS UMR n$^\circ$ 7534), PSL university, Universit\'e Paris-Dauphine, Place de Lattre de Tassigny, 75775 Paris 16, France}
\email{jonteweixler@t-online.de}

\date{\today}

\begin{document}
\begin{abstract}
On the two-dimensional Euclidean space, we study a spinorial analogue of the Caffarelli-Kohn-Nirenberg inequality involving weighted gradient norms. This (SCKN) inequality is equivalent to a spinorial Gagliardo-Nirenberg type interpolation inequality on a cylinder as well as to an interpolation inequality involving Aharonov-Bohm magnetic fields, which was analyzed in a paper of 2020. We examine the symmetry properties of optimal functions by linearizing the associated functional around radial minimizers. We prove that the stability of the linearized problem is equivalent to the positivity of a $2\times2$ matrix-valued differential operator. We study the positivity issue via a combination of analytical arguments and numerical computations. In particular, our results provide numerical evidence that the region of symmetry breaking extends beyond what was previously known, while the threshold of the known symmetry region is linearly stable. Altogether, we obtain refined estimates of the phase transition between symmetry and symmetry breaking. Our results also put in evidence striking differences with the three-dimensional (SCKN) inequality that was recently investigated.

\end{abstract}
\keywords{Caffarelli-Kohn-Nirenberg inequalities; Hardy-Sobolev inequalities; interpolation inequalities; Aharonov-Bohm magnetic fields; symmetry; optimal constants; symmetry breaking; linear instability; non-scalar variational problem; spinors}

\subjclass[2020]{Primary: \msc{35B06}; Secondary: \msc{26D10}, \msc{81Q10}.}

\date{\today. {\em File:} \textsf{\jobname.tex}}
\maketitle
\thispagestyle{empty}

\section{Introduction}\label{Sec:Intro}

With $\Vec{\sigma}\coloneqq (\sigma_1,\sigma_2)^T$ where $\sigma_1$ and $\sigma_2$ denote the first two of the three Pauli matrices
\[
\sigma_1=\left(\begin{array}{cc}0 & 1 \\ 1 & 0 \end{array}\right)\,,\quad\sigma_2= \left(\begin{array}{cc} 0 & -i \\ i & 0 \end{array}\right)\,,\quad\sigma_3= \left(\begin{array}{cc} 1 & 0 \\ 0 & -1 \end{array}\right)\,,
\]
let us consider the \emph{spinorial Caffarelli-Kohn-Nirenberg inequality}
\begin{equation}
\label{SCKN}
\tag{SCKN}
    \int_{\mathbb{R}^2} \frac{|\Vec{\sigma}\cdot\nabla \varphi|^2}{|x|^{2\alpha}} \,dx \geq C_{\alpha,p} \left( \int_{\mathbb{R}^2} \frac{| \varphi|^p}{|x|^{\beta p}} \,dx \right)^\frac{2}{p}
\end{equation}
for spinor valued functions $\varphi : \mathbb{R}^2 \to \mathbb{C}^2$. Here $C_{\alpha,p}$ denotes the best possible constant and the parameters $\alpha\in\R\setminus\Z$, $\beta$ and $p$ are such that
\[
\beta < \alpha \leq \beta+1\quad\mbox{and}\quad p=\frac{2}{\beta-\alpha}\,.
\]
The inequality holds for smooth compactly supported spinors and, by density, for any spinor in the natural Sobolev space associated with~\eqref{SCKN}.
Via a suitable Emden-Fowler type coordinate transformation the inequality transforms into a \emph{logarithmic  Caffarelli-Kohn-Nirenberg inequality} for spinor valued functions $\phi \in \mathrm{H}^1(\mathbb{R}\times \mathbb{S}^1,\mathbb{C}^2)$:
\begin{equation}
\label{CKNSlog}
\tag{SCKN$_{\text{log}}$}
    \int_{\mathbb{R}}\int_{\mathbb{S}^1}\( \left|\partial_s \phi(s,\theta)\right|^2+\left|(\alpha -i\sigma_3 \partial_\theta) \phi(s,\theta) \right|^2\) \,ds\,d\theta \geq C_{\alpha,p} \left(\int_{\mathbb{R}}\int_{\mathbb{S}^1}|\phi(s,\theta)|^p \,ds\,d\theta\right)^\frac{2}{p},
\end{equation}
where $\sigma_3$ is the third Pauli matrix and $\theta$ is the angular variable.
As can be seen by visual inspection, all functionals appearing in the inequality are invariant under rotations of the coordinate system around the origin. 
It is therefore intuitive to expect the optimal functions of the inequality to exhibit the same invariance.

The scalar, real valued version of~\eqref{SCKN} reads as
\be{CKN}\tag{CKN}
\int_{\R^d} \frac{|\nabla u|^2}{|x|^{2\,\mathsf a}}\,dx \ge\mathcal C_{\mathsf a,\mathsf b,d}\left(\int_{\R^d} \frac{|u|^p}{|x|^{\mathsf b\,p}}\,dx \right)^{2/p}
\ee
where $d\ge2$ is an integer, $\mathsf a$, $\mathsf b\in\R$ satisfy $\mathsf a\le\mathsf b\le\mathsf a+1$, $\mathsf b>\mathsf a$ if $d=2$ and, to ensure scale invariance, $p=(2\,d)/\big(d-2-2\,(\mathsf b-\mathsf a)\big)$. Inequality~\eqref{CKN} is known as a \emph{Caffarelli-Kohn-Nirenberg inequality}. This inequality goes back to~\cite{MR0136692,MR768824}. A central issue for~\eqref{CKN} has been the study of the symmetry of the optimal functions with, among many other contributions, the results of \cite{MR1794994,MR1973285,Dolbeault_2016}. In dimension $d=2$, optimal functions are radial if and only if
\[
\mathsf b\ge\mathsf b_{\rm FS}(\mathsf a):=\mathsf a-\frac{\mathsf a}{\sqrt{1+\mathsf a^2}}\,.
\]
For complex valued functions, for instance in the presence of magnetic fields, and for spinors, which are fundamental for applications in quantum mechanics, there are very few results on symmetry. Obviously, symmetrization methods cannot apply directly. In presence of a constant magnetic field, Avron, Herbst and Simon in~\cite{AHS} proved that the ground state of the hydrogenic atom has cylindrical symmetry and Erd\"os in~\cite{Erdos96} established a Faber-Krahn inequality for the Schr\"odinger operator: assuming a homogeneous Dirichlet boundary condition, the disk yields the smallest ground state energy among domains of equal area. Bonheure, Nys and van Schaftingen \cite{MR3926043} showed symmetry in some nonlinear variational problems involving a small constant magnetic field. Another instance of a sharp functional inequality for non-scalar objects (vector fields and spinor fields) is the sharp criterion for zero modes of the Dirac equation in presence of a magnetic field obtained in~\cite{Frank_2024}: the optimizers have the structure that we would like to derive in~\eqref{SCKN}. In the case of an Aharonov-Bohm magnetic field, some interpolation inequalities can be reduced to~\eqref{CKN}, which gives rise to the partial symmetry and symmetry breaking results of~\cite{dolbeault_ckn_2025}. As we shall see later, this problem is in some cases equivalent to~\eqref{SCKN}. Taking advantage of this equivalence, we will numerically refine the estimates of~\cite{bonheure_symmetry_2020}. The counterpart of~\eqref{SCKN} in dimension $d=3$ was recently studied in~\cite{dolbeault_ckn_2025} and involves a more complicated notion of symmetry than simple radial symmetry. Although several results are directly adapted from~\cite{dolbeault_ckn_2025}, qualitative results in dimension $d=2$ significantly differ.

\medskip\noindent\emph{Definitions}. In order to state our results, we have to give a definition of \emph{symmetry} adapted to~\eqref{SCKN}. We shall say that a spinor $\phi=(\varphi_1,\varphi_2)^T : \mathbb{R}_+\times [0,2\pi)\to \mathbb{C}^2 $ is \emph{symmetric} if and only if there exists a $k\in\mathbb{Z}$ such that
\[
(r,\theta)\mapsto\begin{pmatrix}
            e^{ik\theta}\varphi_1(r,\theta)\\ e^{-ik\theta}\varphi_2(r,\theta)
        \end{pmatrix}
\]
is independent of $\theta$. Equivalently, a spinor is symmetric if and only if it is a pure eigenstate to the angular momentum operator $-i\sigma_3 \partial_\theta$. With this definition, we shall say that there is \emph{symmetry} if equality in~\eqref{SCKN} is achieved by a symmetric spinor, and \emph{symmetry breaking} if this is not the case.

\medskip\noindent\emph{Reduction of the parameter space}.
    For symmetry \emph{versus} symmetry breaking issues, the parameter space of $(\alpha,p)\in \mathbb{R}\times (2,\infty)$ can be reduced to $(0,1/2)\times(2,\infty)$ and the search for a radial symmetric minimizer. To see this, first notice that transforming a function to its complex conjugate $\phi \to \overline{\phi}$ transforms the \eqref{CKNSlog} for $(\alpha,\,p)$ into the one for $(-\alpha,p)$. Secondly, for a $k\in\mathbb{Z}$ the transformation $\phi \to e^{ik\theta}\phi $ transforms the \eqref{CKNSlog} for $(\alpha,\, p)$ into the one for $(\alpha +k,\,p)$.
    
\medskip\noindent\emph{Interpolation inequalities for Aharonov--Bohm magnetic fields}.
Let $C_{\alpha,p}$ be the best constant in the \eqref{SCKN} inequality.
Assume that $C_{\alpha,p}^{\rm AB}$ is the best constant in the inequality
\begin{equation}\label{AB}\tag{AB}
    \int_{\R^2} |(-i\nabla-\alpha A)\psi|^2\,dx \geq C_{\alpha,p}^{\rm AB} \left( \int_{\R^2} \frac{|\psi|^p}{|x|^2} \,dx \right)^{2/p}
\end{equation}
for some parameters $\alpha$ and $p$, where
$$
A(x) = \frac1{|x|^2}\,\begin{pmatrix}
       x_2\\-x_1
    \end{pmatrix}
$$
is the Aharonov--Bohm vector potential.
This inequality valid for all sufficiently regular complex-valued functions $\psi$ on $\R^2$. The constant $C_{\alpha,p}^{\rm AB}$ was studied by Bonheure, Dolbeault, Esteban, Laptev and Loss in~\cite[Case $\lambda=0$]{bonheure_symmetry_2020}.
\begin{thm}\label{Thm:Main} Let $(\alpha,p)\in(0,1/2)\times(2,+\infty)$. Inequality \eqref{SCKN} is equivalent to~\eqref{AB} and
	$$
	C_{\alpha,p} = C_{\alpha,p}^{\rm AB} \,.
	$$
	Moreover, optimizers are related by $\phi^\#(s,\theta)=\psi^\#(s,\theta) \chi_0$ for some constant spinor $\chi_0\in\C^2$. As a consequence, there is symmetry in~\eqref{SCKN} and $\phi^\#$ does not depend on $\theta$ if and only if there is symmetry in~\eqref{AB} and $\psi^\#$ does not depend on $\theta$.
\end{thm}
As a consequence of Theorem~\ref{Thm:Main} and~\cite{bonheure_symmetry_2020}, we have the following result for \eqref{SCKN}.
\begin{corollary}\label{Cor:Main}
The following holds:
\begin{enumerate}
    \item For every $\alpha\in(0,1/2)$ and $p > 2$, $C_{\alpha,p}$ is positive, Inequality \eqref{SCKN} admits an optimizer, and $\lim_{\alpha\to0_+}C_{\alpha,p}=0$.
    \item There exists a region $S\subset (0,1/2)\times (2,\infty)$:
    \begin{equation*}
        S\subseteq \left\{(\alpha,p):\: p<\frac 2{|\alpha|}\sqrt{1-3\,\alpha^2}\right\}
    \end{equation*} 
    where all optimizers of \eqref{SCKN} are radially symmetric.
    \item There exists a region $B\subset (0,1/2)\times (2,\infty)$:
    \begin{equation*}
        B\subseteq \left\{(\alpha,p):\: 8\left(\sqrt{p^4-\alpha^2(p-2)^2(p+2)(3p-2)}+2\right)<\alpha^2 (p-2)^3(p+2)+4p(p+4) \right\}
    \end{equation*} where symmetry breaking occurs for \eqref{SCKN}.
    \item There exists a function $p\mapsto\alpha(p): (2,\infty)\to (0,1/2)$ such that symmetry holds if $\alpha\in\big(0,\alpha(p)\big]$ and there is symmetry breaking if $\alpha\in\big(\alpha(p),1/2\big)$.
    \item Denote by $\phi_*$ the optimizer of \eqref{CKNSlog} among all radially symmetric functions. The linear stability of \eqref{CKNSlog} around it is equivalent to the condition, that the operator $A$ acting on $\mathrm{H}^1(\mathbb{R})\times \mathrm{H}^1(\mathbb{R})$ and defined by
    \begin{equation}\label{A}
        A \coloneqq \begin{pmatrix}
        (-\partial_x^2+(1+\alpha)^2-\frac{p}{2}\,|\phi_*|^{p-2} ) &- \frac{p-2}{2}\,|\phi_*|^{p-2} \\
        - \frac{p-2}{2}\,|\phi_*|^{p-2} & (-\partial_x^2+(1-\alpha)^2-\frac{p}{2}\,|\phi_*|^{p-2} ) 
    \end{pmatrix}
    \end{equation}
    is positive semi-definite.
\end{enumerate}
\end{corollary}
\noindent These properties are straightforward consequences of Theorem~\ref{Thm:Main} except for (5) which is new. See Theorem~\ref{Equivalence for Lin Instab Theorem} below for details.
The results on symmetry and symmetry breaking regions are illustrated in Figure~\ref{fig:stability_regions}.
    
\medskip\noindent\emph{Summary of the numerical results}.
The operator $A$ defined by~\eqref{A} can be represented as an infinite dimensional matrix $M$ in a Gegenbauer polynomial basis. Via a finite dimensional truncation of the matrix $M$ and a numerical computation of the lowest eigenvalue of this truncation, we were able to obtain the following numerical results:
    \begin{itemize}
        \item The threshold between linear instability and linear instability lies strictly between the established regions of symmetry and symmetry breaking.
        \item The symmetry breaking region is larger than established in \cite{bonheure_symmetry_2020}.
    \end{itemize}
\noindent Beyond Theorem~\ref{Thm:Main}, these results are the main contribution of the current paper. See Figures \ref{fig:sign40stability_regions} and \ref{fig:sign40stability_regions} for illustration of typical computations and Figure \ref{fig:3dstability_regions} for a more detailed analysis of the values of lowest eigenvalue. 

\medskip
This paper is organized as follows. Section \ref{From 2D-spinor to Aharonov--Bohm} is devoted to the proof of Theorem~\ref{Thm:Main}. In Section \ref{Linear Instability Region} we consider linear instability results, prove that it implies symmetry breaking and establish Result (5) of Corollary~\ref{Cor:Main}. Section \ref{Reformulation in a Gegenbauer Polynomial Basis} is intended to a characterization of the spectrum of $A$ given by~\eqref{A} using a matrix in a Gegenbauer polynomial basis.
Building upon this result, we explain the method for our numerical results in Section \ref{Numerical Method} and discuss its robustness and some further results.

\section{From 2D-spinor to Aharonov--Bohm}\label{From 2D-spinor to Aharonov--Bohm}

This section is devoted to the proof of our main analytical result.

\begin{proof}[Proof of Theorem~\ref{Thm:Main}] We divide the proof in three steps.

\medskip\noindent$\bullet$
	\emph{Step 1.} Let us denote by $C_{\alpha,p}^{\rm scalar}$ the best constant in the inequality
	$$
	\int_{\R^2} \frac{|(\partial_1+i\partial_2)\psi|^2}{|x|^{2\alpha}}\,dx \geq C_{\alpha,\beta}^{\rm scalar} \left( \int_{\R^2} \frac{|\psi|^p}{|x|^{\beta p}} \,dx \right)^{2/p} \,,
	$$
	for \emph{scalar} (that is, $\C$-valued) fields on $\R^2$. We claim that
	$$
	C_{\alpha,p} = C_{\alpha,p}^{\rm scalar} \,.
	$$
	
	We write $\phi = (\varphi_1,\varphi_2)^{\rm T}$ with scalar fields $\varphi_1$ and $\varphi_2$. By definition of the Pauli matrices, we have
	$$
	\int_{\R^2} \frac{|\sigma\cdot\nabla\phi|^2}{|x|^{2\alpha}}\,dx = \int_{\R^2} \frac{|(\partial_1+i\partial_2)\varphi_1|^2+|(\partial_1-i\partial_2)\varphi_2|^2}{|x|^{2\alpha}}\,dx
	$$
	Noting that $|(\partial_1-i\partial_2)\varphi_2| = |(\partial_1+i\partial_2)\overline{\varphi_2}|$, we find that
	$$
	\int_{\R^2} \frac{|\sigma\cdot\nabla\phi|^2}{|x|^{2\alpha}}\,dx \geq C_{\alpha,\beta}^{\rm scalar} \left( \left( \int_{\R^2} \frac{|\varphi_1|^p}{|x|^{\beta p}} \,dx \right)^{2/p} + \left( \int_{\R^2} \frac{|\overline{\varphi_2}|^p}{|x|^{\beta p}} \,dx \right)^{2/p} \right).
	$$
	By the triangle inequality in $\mathrm L^{p/2}$, we have
	\[
	\left( \int_{\R^2} \frac{|\varphi_1|^p}{|x|^{\beta p}} \,dx \right)^{2/p} + \left( \int_{\R^2} \frac{|\overline{\varphi_2}|^p}{|x|^{\beta p}} \,dx \right)^{2/p}
		 \geq \left( \int_{\R^2} \frac{(|\varphi_1|^2 + |\overline{\varphi_2}|^2)^{p/2}}{|x|^{\beta p}} \,dx \right)^{2/p} 
		 = \left( \int_{\R^2} \frac{|\phi|^p}{|x|^{\beta p}} \,dx \right)^{2/p}.
		\]
	This proves that $C_{\alpha,\beta}\geq C_{\alpha,\beta}^{\rm scalar}$. The reverse inequality is trivial. 
	Equality in the triangle inequality implies that $\phi$ for $C_{\alpha,p}$ are necessarily of the form
	$$
	\phi = \psi \chi_0 \,.
	$$
	where $\chi_0 \in \C^2$ is a constant spinor.
	Since optimizers achieve equality everywhere, this shows that optimizers $\phi$ for $C_{\alpha,p}$ are necessarily of that form	where $\psi$ is an optimizer for $C_{\alpha,p}^{\rm scalar}$.
	
\medskip\noindent$\bullet$
\emph{Step 2.} For a spinor field $\psi$ on $\R^2$ in the natural Sobolev space associated to \eqref{SCKN} we introduce an Emden Fowler coordinate transfomation and define a transformed spinor field \mbox{$\phi \in \mathrm{H}^1(\R\times\Sph^1,\,\mathbb{C}^2)$} by
	$$
	\psi(r\cos\theta,r\sin\theta) = r^\alpha \phi(\ln r,\theta) \,.
	$$
	Then, taking the relation between $\alpha,\beta$ and $p$ into account,
	$$
	\int_{\R^2} \frac{|\psi|^p}{|x|^{\beta p}}\,dx = \iint_{\R\times\Sph^1} |\phi|^p \,ds\,d\theta \,.
	$$
	Moreover, 
	$$
	\int_{\R^2} \frac{|(\partial_1+i\partial_2)\psi|^2}{|x|^{2\alpha}}\,dx
	= \iint_{\R\times\Sph^1} |\partial_s \phi + \alpha\phi - i\partial_\theta\phi|^2\,ds\,d\theta \,.
	$$
	We claim that the right side is equal to
	$$
	\iint_{\R\times\Sph^1} \left( |\partial_s \phi|^2 + |(-i\partial_\theta +\alpha)\phi |^2 \right)ds\,d\theta \,.
	$$
	To see this, we expand
	$$
	|\partial_s \phi + \alpha\phi - i\partial_\theta\phi|^2 = |\partial_s\phi|^2 + |(-i\partial_\theta+\alpha)\phi|^2 + 2\re\langle \overline{\partial_s\phi},\ (-i\partial_\theta+\alpha)\phi\rangle_{\C^2}.
	$$
	We need to show that the integral of the mixed term vanishes for all \mbox{$\phi \in \mathrm{H}^1(\R\times\Sph^1,\,\mathbb{C}^2)$}. First, we observe that
	$$
	\re \int_\R \langle{\partial_s \phi},\ \phi\rangle_{\C^2}\,ds = \frac{1}{2} \re \int_\R \partial_s (|\phi|^2)\,ds = 0 \,.
	$$
	Next, by integrating by parts twice (once in $s$ and once in $\theta$), we find
	$$
	\iint_{\R\times\Sph^1} \langle{\partial_s \phi}, \ \partial_\theta\phi \rangle_{\C^2} \,ds\,d\theta = \iint_{\R\times\Sph^1} \langle {\partial_\theta\phi}, \ \partial_s\phi \rangle_{\C^2} \,ds\,d\theta \,.
	$$
	Multiplying by $-i$ and taking the real part gives
	\begin{align*}
		\re \iint_{\R\times\Sph^1} \langle{\partial_s \phi}, \ (-i\partial_\theta\phi)\rangle_{\C^2} \,ds\,d\theta & =  \re \iint_{\R\times\Sph^1} \ \langle{ (-i \partial_\theta\phi)}, \ \partial_s\phi  \rangle_{\C^2}\,ds\,d\theta \\
		& = - \re \iint_{\R\times\Sph^1} \langle{\partial_s \phi}, \ (-i\partial_\theta\phi)\rangle_{\C^2} \,ds\,d\theta \,.
	\end{align*}
	Thus, also the second mixed term vanishes.

	To summarize, we have shown that the $C_{\alpha,\beta}^{\rm scalar}$ is the best constant in the spinorial inequality
	$$
	\iint_{\R\times\Sph^1} \left( |\partial_s \phi|^2 + |(-i\partial_\theta +\alpha)\phi |^2 \right)ds\,d\theta \geq C_{\alpha,\beta} \left( \iint_{\R\times\Sph^1} |\phi|^p \,ds\,d\theta \right)^{2/p}.
	$$
	
\medskip\noindent$\bullet$
\emph{Step 3.} We finally show that the Aharonov--Bohm inequality can be brought into the same form as the inequality at the end of the previous step. This follows by similar computations as in the previous step. We assume that $\psi$ and $\tilde\phi$ are related by
	$$
	\psi(r\cos\theta,r\sin\theta) = \tilde\phi(\ln r,\theta) \,.
	$$
	Then
	$$
	\int_{\R^2} |(-i\nabla - \alpha A)\psi|^2\,dx = \iint_{\R\times\Sph^1} \left( |\partial_s\tilde\phi|^2 + |(-i\partial_\theta +\alpha)\tilde\phi|^2\right)ds\,d\theta
	$$
	and
	$$
	\int_{\R^2} \frac{|\psi|^p}{|x|^2}\,dx = \iint_{\R\times\Sph^1} |\tilde\phi|^2 \,ds\,d\theta \,.
	$$
	This completes the proof of the equivalence.		
\end{proof}

\section{Linear Instability Region}
\label{Linear Instability Region}

By considering linear perturbations of the optimizer of~\eqref{SCKN} restricted to symmetric spinors as, \emph{e.g.}, in~\cite{MR1973285}, we prove symmetry breaking if we can find a negative eigenvalue. Let us start with the inequality restricted to symmetric spinors.
\begin{lem}[Optimizer among Radially Symmetric Functions]
    The optimal function for \eqref{CKNSlog} among all symmetric functions $\phi \in \mathrm{H}^1(\mathbb{R}\times \mathbb{S}^1,\mathbb{C}^2) $ is given up to a multiplicative constant and translation by
    \begin{equation*}
        \phi_*(s) \coloneqq \left(\frac{p\:\alpha^2}{2}\right)^\frac{1}{p-2} \cdot \left( \cosh\left( \frac{p-2}{2} \alpha s\right) \right)^{-\frac{2}{p-2}} \cdot \chi
    \end{equation*}
    where $\chi \in \mathbb{C}^2$ is a constant unit spinor.
    Furthermore, the optimal constant among radially symmetric functions in $\mathrm{H}^1(\mathbb{R}\times \mathbb{S}^1,\mathbb{C}^2) $ is given by
    \begin{align*}
        (2\pi)^{\frac{p}{2}-1} C_{\alpha,p}^* = \| \phi_*(s)\|_{L^p(\mathbb{R})}^{p-2}.
    \end{align*}
    Since the set of radially symmetric functions is a subset smaller than the whole set it trivially holds that $C_{\alpha,p}\leq C_{\alpha,p}^*$.
    Furthermore, the absolute value of $\phi_*$
    \begin{equation*}
        \varphi_*(s) \coloneqq|\phi_*(s)|= \left(\frac{p\:\alpha^2}{2}\right)^\frac{1}{p-2} \cdot \left( \cosh\left( \frac{p-2}{2} \alpha s\right) \right)^{-\frac{2}{p-2}}
    \end{equation*}
    is the unique (up to translations) solution to the non linear Schrödinger type one dimensional differential equation
    \begin{equation*}
        -\partial_s^2 \varphi_*(s)+ \alpha^2 \varphi_*(s)- \varphi_*(s)^{p-1}=0\,.
    \end{equation*}
\end{lem}
\begin{proof}
    The proof goes back to Nagy \cite{MR4277}. An accessible introduction taylored to our use case can be found in~\cite[Appendix A]{dolbeault_ckn_2025}.
\end{proof}

    The \emph{deficit functional} with the optimal radial constant for functions $\phi \in \mathrm{H}^1(\mathbb{R}\times \mathbb{S}^1,\mathbb{C}^2)$ is defined as
    \begin{align*}
        \mathcal{F}[\phi]\coloneqq \underbrace{\int_{\mathbb{R}}\int_{\mathbb{S}^1} \left|\partial_s \phi(s,\theta)\right|^2+\left|(\alpha -i\sigma_3 \partial_\theta) \phi(s,\theta) \right|^2 ds d\theta}_{\eqqcolon \mathcal{A}[\phi]} - C_{\alpha,p}^* \underbrace{ \left(\int_{\mathbb{R}}\int_{\mathbb{S}^1}|\phi(s,\theta)|^p ds d\theta\right)^\frac{2}{p}}_{\eqqcolon \mathcal{D}[\phi]}\,.
    \end{align*}
    The quadratic form of the linearized problem for a $\varphi \in \mathrm{H}^1(\mathbb{R}\times \mathbb{S}^1,\mathbb{C}^2)$ is defined as
    \begin{equation*}
    \mathcal{Q}[\varphi]\coloneqq \lim_{\varepsilon\to 0} \frac{1}{\varepsilon^2} \left(\mathscr{F}[\phi_*+\varepsilon\varphi]-\mathscr{F}[\phi_*]\right).
\end{equation*}

\noindent
By Taylor expansion up until second order in $\varepsilon$ we report that
\begin{multline*}
    \mathcal{Q}[\varphi]= \norm{\partial_s \varphi}^2_{L^2(\mathbb{R}\times \mathbb{S}^1)}+\norm{(-i\sigma_3\partial_\theta +\alpha) \varphi}_{L^2(\mathbb{R}\times \mathbb{S}^1)}^2\\
    -\Bigg[ \iint_{\mathbb{R}\times \mathbb{S}^1} |\phi_*|^{p-2} |\varphi|^2dsd\vartheta + (p-2) \iint_{\mathbb{R}\times \mathbb{S}^1} |\phi_*|^{p-4} |\Re(\langle \phi_*,\varphi \rangle )|^2 ds d\vartheta \\
    -\frac{p-2}{\norm{\phi_*}_{L^p(\mathbb{R}\times \mathbb{S}^1)}^p} \left( \iint_{\mathbb{R}\times \mathbb{S}^1} |\phi_*|^{p-2} \Re(\langle \phi_*,\varphi \rangle ) ds d\vartheta\right)^2 \Bigg]
\end{multline*}
using the normalized measure $d\vartheta$ on the circle in all norms.
If we obtain for some spinor \mbox{$\varphi \in \mathrm{H}^1(\mathbb{R}\times \mathbb{S}^1,\mathbb{C}^2)$}  that $\mathcal{Q}[\varphi]<0$, we shall say that \emph{the linearized problem linearly is unstable} around the radially symmetric optimizer.
\begin{lem}
If the linearized problem linearly is unstable around the radially symmetric optimizer, then the optimal function of \eqref{CKNSlog} is not radially symmetric.
\end{lem}
\begin{proof}
    Since $\phi_*$ is the radial optimizer, we know by definition that $\mathcal{F}[\phi_*]=0$.
    By our assumption of linear instability we can find a small $\varepsilon>0$ such that $$ \mathcal{F}[\phi_*+\varepsilon \varphi]-\underbrace{\mathcal{F}[\phi_*]}_{=0}= \mathcal{A}[\phi_*+\varepsilon \varphi]- C_{\alpha,p}^* \cdot \mathcal{D}[\phi_*+\varepsilon \varphi] <0\,.$$
    Since $\phi_*+\varepsilon \varphi \in \mathrm{H}^1(\mathbb{R}\times \mathbb{S}^1,\mathbb{C}^2)$, the \eqref{CKNSlog} holds, which yields a contradiction. The optimal constant must therefore be smaller than the optimal radial constant. Hence, the radial optimizer cannot be the global optimizer.
\end{proof}

A goal is to identify the whole region in parameter space $(\alpha,p)$ for which the quadratic form $\mathcal{Q}$ is not positive semi-definite, \emph{i.e.}, to characterize explicitly the region of linear instability for $\mathcal{Q}$. For this we establish the following result:
\begin{thm}
\label{Equivalence for Lin Instab Theorem}
    The parameter region $(\alpha,p)$ of linear instability of $\mathcal{Q}$ is equivalent to the parameter region $(\alpha,p)$ for which the operator $A$ defined by~\eqref{A} is not positive semi-definite.
\end{thm}
\begin{proof}[Sketch of the Proof]
    The proof relies on an $\mathrm{L}^2$ spherical harmonics decomposition. One can show, that the only relevant directions for linear instability is the contribution of the spin-up component with angular momentum equal to $1$ and the contribution of the spin-up component with angular momentum equal to $-1$. The result then follows from a variational argument.
\end{proof}

We use the following two test functions to test for linear instability and therefore to obtain symmetry breaking regions:
\begin{itemize}
    \item $\phi_1(s,\theta)= \varphi_1(s) \cdot \begin{pmatrix}
        0\\ 1
    \end{pmatrix}$
    \item $\phi_2(s,\theta)= \varphi_2(s) \cdot \begin{pmatrix}
        \sqrt{1-t^2}   \\ t
    \end{pmatrix}$.
\end{itemize}
Transformed back into nonradial functions $\phi \in \mathrm{H}^1(\mathbb{R}\times \mathbb{S}^1,\mathbb{C}^2)$ they correspond to the following:
\begin{itemize}
    \item $\phi_1(s,\theta)= \varphi_1(s) \cdot \begin{pmatrix}
        e^{-i\theta}\\ 0
    \end{pmatrix}$
    \item $\phi_2(s,\theta)= \varphi_2(s) \cdot \begin{pmatrix}
        t \cdot e^{i\theta}+ \sqrt{1-t^2} \cdot e^{-i\theta}\\ 0
    \end{pmatrix}$ for a $t\in[0,1]$ which will be optimzed over later.
\end{itemize}
The above two test functions are only informed guesses for the direction $\varphi$ in which the quadratic form $\mathcal{Q}$ is the lowest. They do not seem to be optimal.
\begin{figure}[ht]
\begin{center}
\includegraphics[width=8cm]{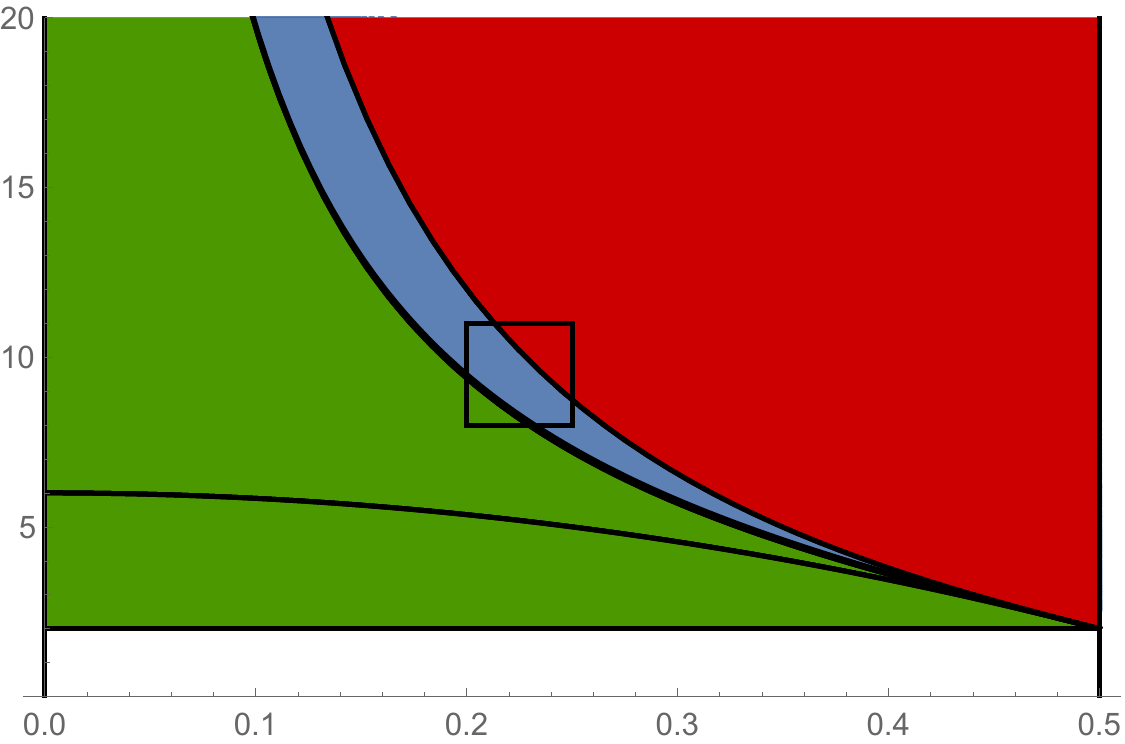}\hspace*{1cm}\includegraphics[width=6cm]{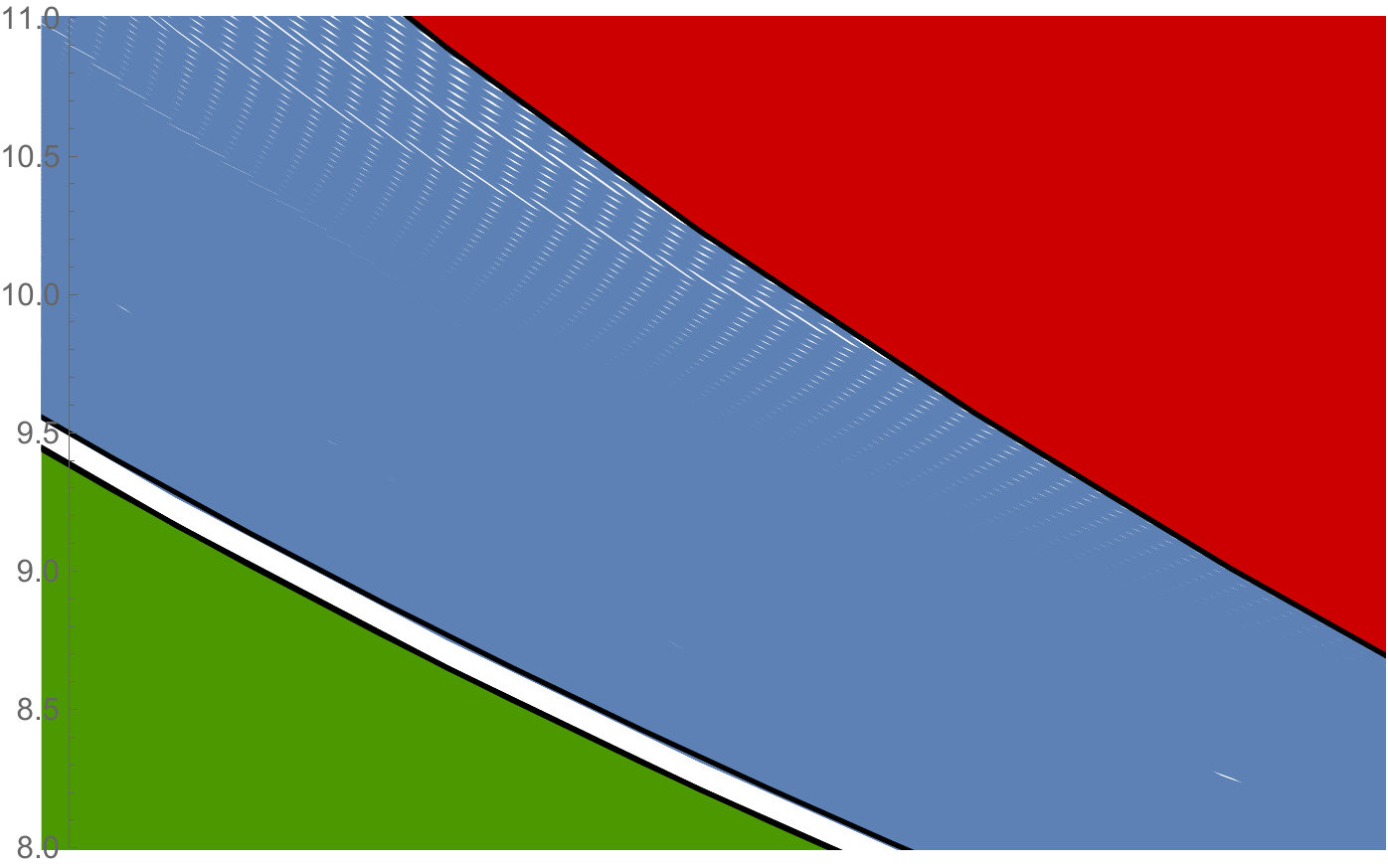}
\end{center}
\caption{\small\emph{Symmetry (green, established in \cite{bonheure_symmetry_2020}) and symmetry breaking (red, blue) regions in the $(\alpha,p)$ representation, \hbox{$\alpha \in \left(0,1/2\right),\, p \in (2,\infty)$.} The range in which symmetry \emph{versus} symmetry breaking is not decided is the tiny white gap shown in the enlargement (right) corresponding to the black rectangle (left). It is known from \cite{bonheure_symmetry_2020} that there exists one seperating function $\alpha(p)$ seperating the two regions. No regularity for this function is known.
The red and blue regions correspond to the two different Ansätze used to test the linearized problem. The black line seemingly bounding the blue region is the region of linear instability given in~\cite{bonheure_symmetry_2020} uses a different Ansatz for the linearized problem, which is apparently equivalent to ours.}}
\label{fig:stability_regions}
\end{figure}
In both cases $i=1,2$, the quadratic form associated to $A$ takes the form $\langle \phi_i,A\phi_i \rangle_{\mathrm{L}^2(\mathbb{R},\mathbb{C}^2)} = \langle \varphi_i, P_i \varphi_i \rangle_{\mathrm{L}^2(\mathbb{R},\mathbb{C})} $ for a one dimensional Pöschl-Teller operator $P_i$, for which the lowest eigenvalue is explicitly known (see for example p.74 of \cite{Landau-Lifschitz-67} or \cite{poschl1933bemerkungen}).
For $i=1$ we choose $\varphi_i(s)$ to be the eigenfunction to the lowest eigenvalue of the corresponding Pöschl-Teller operator $P_1$ and obtain as a condition for the non-negativity of the lowest eigenvalue
\begin{equation*}
    p\leq \frac{\sqrt{9\,\alpha^2-14\,\alpha+5}-\alpha +1}\alpha\,.
\end{equation*}
The region of symmetry breaking established in this way is colored in red in Figure \ref{fig:stability_regions}. For $i=2$ we take $\varphi_2(s)$ to be the eigenfunction to the lowest eigenvalue of the corresponding Pöschl-Teller operator $P_2$. This operator and therefore also the eigenfunction and the corresponding eigenvalue are now dependent on $t$. As a condition on the positivity of the lowest eigenvalue we obtain:
\begin{equation*}
    \frac{1}{16}\left(2\,\alpha+\sqrt{\alpha^2\left(4\,p^2+8\,p\,(p-2)\,t\,\sqrt{1-t^2}+(p-2)^2\right)}-\alpha\,p\right)^2 \leq  \alpha ^2+\alpha\left(4\,t^2-2\right)+1\,.
\end{equation*}
One can now plot the curves corresponding to different values of $t\in[0,1]$. The region of symmetry breaking obtained in this way are colored in blue in \ref{fig:stability_regions}. 
In \cite{bonheure_symmetry_2020} the authors establish symmetry breaking via a visually, but perhaps not essentially different Ansatz. 
The threshold from their paper seems to be the envelope of the blue region.
The threshold is plotted in Figure \ref{fig:stability_regions} as the black line at the interface between the blue and white region. The numerics suggest the conjecture, that the Ansatz used in \cite{bonheure_symmetry_2020} to obtain linear instability is essentially equivalent to the Ansatz $\phi_2$ used above to obtain the blue region.

\section{Reformulation in a Gegenbauer Polynomial Basis}
\label{Reformulation in a Gegenbauer Polynomial Basis}

The goal of this section is to lay the theoretical groundwork to make the question of positive semi-definiteness of $A$ defined in \ref{Equivalence for Lin Instab Theorem} amenable to numerical analysis.
We find two avenues of approach to a numerical spectral analysis of $A$. One would be a Birmann-Schwinger approach as done for example in \cite{dolbeault2023kellerliebthirringestimateseigenvalues}. We are currently implementing this.
The second one relies on reformulating the differential operator $A$ as an infinite dimensional matrix in a basis of special functions. For this we use Gegenbauer polynomials as follows:
\\
The differential operator $A$ from Theorem \ref{Equivalence for Lin Instab Theorem} can be identified with a $2\times 2$ block matrix $M= \begin{pmatrix}
    M_1 & M_2\\
    M_3 & M_4
\end{pmatrix} $ where $M_i \in \mathbb{R}^{\mathbb{N}_0\times \mathbb{N}_0}$.
First define the numbers:
\begin{align*}
    \eta_k^1 &\coloneqq\frac{1}{2(k+\lambda)}\left( \frac{(k+1)(k+2\lambda))}{2(k+1+\lambda)}+\frac{(k-1+2\lambda)k}{2(k-1+\lambda)}\right)\,,\\[4pt]
    \eta_k^2&\coloneqq \frac{(k+1)(k+2)}{4(k+\lambda)(k+1+\lambda)}\,,\\[4pt]
    \eta_k^0 &\coloneqq \frac{(k-1+2\lambda)(k-2+2\lambda)}{4(k+\lambda)(k-1+\lambda)}\,.
\end{align*}

Define the following matrices (infinite dimensional) written in the Gegenbauer polynomial basis $\left(C^\lambda_k(z)\right)_{k\in\mathbb{N}_0} $ (using the definitions $\lambda=\frac{n-3}{2}$ and $n:=\frac{2\,p}{p-2}$) :
\begin{align*}
    G &\coloneqq \begin{pmatrix} 1-\eta_0^1 & 0 & -\eta_2^0 & 0& \dots & \\
    0 & 1-\eta_1^1 &0 & -\eta_3^0 & 0 & \dots \\
    -\eta^2_0 & 0 & 1-\eta_2^1 &0 & -\eta_4^0 &  0& \dots \\
    0&-\eta^2_1 & 0 & 1-\eta_3^1 &0 & -\eta_5^0 &  0& \dots\\
    \vdots & \ddots & \ddots & 0 & \ddots & 0 & \ddots &0 & \dots       
    \end{pmatrix}= (1-z^2)\,,\\
    A &\coloneqq \begin{pmatrix}
        0 & 0 & 2\lambda & 0 & 2\lambda & 0 &  \hdots \\
        0 & 1 & 0& (2+2\lambda) & 0 & (2+2\lambda) & \hdots \\
        0& 0& 2 & 0 & (4+2\lambda) & 0  & (4+2\lambda) & \ddots \\
        0 & \vdots & 0 & \ddots & 0 & \ddots & 0 &\ddots
    \end{pmatrix} = z\,\frac{d}{dz}\,,
    \end{align*}
 the diagonalized ultraspherical operator for parameter $\lambda$
 \begin{align*}
  B &\coloneqq \text{diag}(k(k+2\lambda))= \begin{pmatrix}
        0 & 0&0 & 0 & \dots\\
        0 & (1+2\lambda) &0 & 0& \dots \\
        0& 0 & 2(2+2\lambda) &0 & \dots \\
        0& 0& 0 & 3(3+2\lambda) &0 & \dots \\
        \vdots & \vdots & \ddots &\ddots & \ddots & \ddots
    \end{pmatrix}\,,
     \end{align*}
and
\[
    N \coloneqq \text{diag}\left(\frac{\pi\: 2^{1-2\lambda} \Gamma(k+2\lambda)}{k! (k+\lambda) (\Gamma(\lambda))^2}\right)\,.
\]
Then we use the Gegenbauer polynomial decomposition $\varphi_1(s)\overset{\text{id}}{=}  \Vec{w}_1$ and $\varphi_{2}(s)\overset{\text{id}}{=}  \Vec{w}_2$ to write $M$ as a blockmatrix in the basis $(\Vec{w}_1,\Vec{w}_2)^T
$ as
\begin{multline*}
    M\coloneqq\left(\frac{p \alpha^2}{2}\right)^{\frac{n-4}{2}} \frac{p\alpha}{(p-2)} \begin{pmatrix} N&0\\0&N \end{pmatrix} \Bigg[ \left(\frac{(p-2)\:\alpha}{2}\right)^2 \begin{pmatrix}
      G (B+2A)& 0 \\
      0 & G (B+2A) 
    \end{pmatrix}\\
     +  \begin{pmatrix}
        (1+2\alpha) \mathds{1} &0 \\
        0 &  (1-2\alpha) \mathds{1}
    \end{pmatrix} 
    + \frac{p\:\alpha^2}{4}\begin{pmatrix}
        (2-p)\,G & (2-p) \,G\\
        (2-p)\, G & (2-p)\, G
    \end{pmatrix} \Bigg] \,.
\end{multline*}
\begin{thm}
With the above notation, if there is some $\ket{w}=(\Vec{w}_1,\Vec{w}_2)^T$ such that
\begin{equation*}
    \bra{w}M\ket{w}_{\ell^2(\mathbb{N}_0^2)}<0
\end{equation*}
then~\eqref{SCKN} is linear instable.
\end{thm}
\begin{proof}
Recall that the function $\varphi_*(s)\coloneqq\left(\frac{p\alpha^2}{2}\right)^\frac{1}{p-2} \left(\cosh\left(\frac{p-2}{2} \alpha s\right)\right)^{-\frac{2}{p-2}}$ solves
\begin{equation}
\label{phistar}-\partial_s^2\varphi_*+\alpha^2\,\varphi_*=\varphi_*^{p-1}\,.
\end{equation}
Now define for $\varphi \in \mathrm{H}^1(\mathbb{R)}$ the function $u$ such that $\varphi=u\cdot \varphi_*$.
An expansion of the square then shows that
\begin{align*}
\iR{|\varphi'|^2}=\iR{|\varphi_*\,u'+\varphi_*'\,u|^2}=\iL{|u'|^2}2+\iR{\left(u^2\right)'\,\varphi_*\,\varphi_*'}+\iR{|u|^2\left(\varphi_*'\right)^2}\,.
\end{align*}
Define the measure 
\begin{equation*}
    d\mu_k\coloneqq |\varphi_*|^k\,ds\,.
\end{equation*}
Then, using an integration by parts, we obtain
\[
\iR{|\varphi'|^2}=\iL{|u'|^2}2-\iR{|u|^2\,\varphi_*\,\varphi_*''}
\]
and, using~\eqref{phistar} and $\iR{|\varphi|^2}=\iL{|u|^2}2$,
\[
\iR{|\varphi'|^2}+\alpha^2\iR{|\varphi|^2}=\iL{|u'|^2}2+\iL{|u|^2}p\,.
\]
which we obtained by using $\iR{|\varphi|^p}=\iL{|u|^p}p$.\\
We now want to utilize a change of variables. The change of variables is defined as
\begin{equation*}
    s\mapsto z(s)=\tanh\left(\frac{p-2}{2} \alpha s\right)\,.
\end{equation*} 
It is tailored such that
\[
1-z(s)^2=\frac1{\left(\cosh \left(\frac{p-2}{2} \alpha s\right)\right)^2}=\frac{2}{p\alpha^2}|\varphi_*(s)|^{p-2}\,.
\]
Define the number $n:=\frac{2\,p}{p-2}$ which is in $(2,\,\infty)$. Thereafter, consider the transformed function $w\big(z(s)\big)=u(s)$. By calculation obtain the following relation:
\[
\frac{dz(s)}{ds}= \frac{(p-2)\alpha}{2}{\left(\cosh\left( \frac{p-2}{2}\alpha s\right)\right)^{-2}}
    =\frac{(p-2)\alpha}{2} \frac{2}{p\alpha^2}|\varphi_*(s)|^{p-2}
    =\frac{(p-2)\alpha}{2} (1-z(s)^2)\, .
    \]
Using this relation in the change of variables, we obtain the following transformed integrals
\begin{align*}
&\frac{p\alpha }{(p-2)} \left( \frac{p\alpha^2}{2}\right)^{\frac{n-2}{2}}\isn{|w(z)|^2}{\frac{n-2}2}=\iL{|u(x)|^2}p\,,\\
&\frac{p\alpha }{(p-2)} \left( \frac{p\alpha^2}{2}\right)^{\frac{n-2}{2}}\isn{|w(z)|^p}{\frac{n-2}2}=\iL{|u(x)|^p}p\,,\\
&\frac{p\alpha }{(p-2)} \left( \frac{p\alpha^2}{2}\right)^{\frac{n-4}{2}}\isn{|w(z)|^2}{\frac{n-4}2}=\iL{|u(x)|^2}2\,,\\
&\frac{(p-2) }{p\alpha} \left( \frac{p\alpha^2}{2}\right)^{\frac{n}{2}}\isn{|w'(z)|^2}{\frac n2}=\iL{|u'(x)|^2}2\,.\\
\end{align*}
Now for a $\phi(s)=(\varphi_1,\varphi_{2})^T$, we are interested in the quadratic form $\langle \phi, A \phi \rangle_{\mathrm{L}^2(\mathbb{R}\times\mathbb{C}^2)}$

\begin{align*}
    \frac{p\alpha }{(p-2)} \left( \frac{p\alpha^2}{2}\right)^{\frac{n-4}{2}} \int_{2}^1\Biggl[  &
    \left(\frac{(p-2)\alpha}{2}\right)^2 
    \left\langle 
        \begin{pmatrix} w'_1(z)\\w'_2(z)\end{pmatrix}, 
        \begin{pmatrix} w'_1(z)\\w'_2(z)\end{pmatrix} 
    \right\rangle_{\mathbb{C}^2} 
    (1-z^2)^{\frac{n}{2}} \\
    &+ \frac{p\alpha^2}{2} 
    \left\langle 
        \begin{pmatrix} w_1(z)\\w_2(z)\end{pmatrix}, 
        \left( \mathds{1} - \frac{1}{2}
        \begin{pmatrix} p & (p-2) \\ (p-2) & p \end{pmatrix} \right) 
        \begin{pmatrix} w_1(z)\\w_2(z)\end{pmatrix} 
    \right\rangle _{\mathbb{C}^2} 
    (1-z^2)^{\frac{n-2}{2}} \\
    &+  
    \left\langle 
        \begin{pmatrix} w_1(z)\\w_2(z)\end{pmatrix}, 
        \begin{pmatrix} (1+2\alpha) & 0 \\ 0 & (1-2\alpha) \end{pmatrix} 
        \begin{pmatrix} w_1(z)\\w_2(z)\end{pmatrix} 
    \right\rangle_{\mathbb{C}^2} 
    (1-z^2)^{\frac{n-4}{2}}
    \Biggr]dz
\end{align*}

Now we consider the summands one by one. 
If one writes the transfomed functions $w_i$ in a Gegenbauer polynomial basis as an $\Vec{w_i}\in \ell^2(\mathbb{N}_0)$ vector, then the above quadratic form can be written as $\langle \Vec{w_i}, M \Vec{w_i} \rangle_{\ell^2(\mathbb{N}_0)}$ for a matrix 
\[
M= \begin{pmatrix}
    M_1 & M_2\\
    M_3 & M_4
\end{pmatrix}
\]
where $M_i \in \mathbb{R}^{\mathbb{N}_0\times \mathbb{N}_0}$. The goal is to show the claim above holds and to compute this matrix $M$.
The Gegenbauer polynomial basis used is the one for parameter $\lambda=(n-3)/2$. We will follow the unsual convention in calling these polynomials $C^\lambda_k(z)$ for $k\in \mathbb{N}_0$. Furthermore, for $k<0$ we will use the convention that $C^\lambda_k(z)=0$ to make case distinctions obsolete. Notice that $n>2$ implies $\lambda>-1/2$ which ensures the existence of the Gegenbauer polynomials.
The polynomials are orthogonal with respect to the $\mathrm{L}^2([-1,1],\mathbb{C}, d\mu)$ scalar product with weight $d\mu(z)=(1-z^2)^\frac{n-4}{2}dz $. They are eigenfunctions to the ultraspherical operator 
\begin{align*}
    \mathcal L_{(n-3)/2}\coloneqq -(1-z^2)\frac{d^2}{dz^2}+(n-2)z \frac{d}{dz}
\end{align*}
with eigenvalues $\lambda_k= k(k+n-3) $ for the $C_k^\frac{n-3}{2}(z)$. Therefore it holds that
\begin{align*}
    \isn{|w'|^2}{\frac{n}{2}}= \langle w,\mathcal L_{(n-3)/2}w \rangle_{\mathrm{L}^2([-1,1],\mathbb{C}, d\mu)}+2\int_{2}^{+1} (1-z^2)^{\frac{n-4}{2}}\: \overline{w}\: z\frac{d}{dz} w(z)\: dz\,.
\end{align*}
Now turn the attention to the $ z\frac{d}{dz}$ using classical identities~\cite[(15) and (16)]{weisstein_gegenbauer}, specifically identity :
\begin{align*}
     z\frac{d}{dz} C^\lambda_k(z) &\overset{(15)}{=} k\:  C^\lambda_k(z)+\frac{d}{dz} C^\lambda_{k-1}(z)\\
     &\overset{(16)}{=} k\:  C^\lambda_k(z)+ (k+n-5) \:C^\lambda_{k-2}(z) + z\frac{d}{dz} C^\lambda_{k-2}(z)\\
     &\qquad \qquad \vdots \qquad \vdots \qquad \vdots\\
     &=k\: C^\lambda_k(z) + \sum_{j\in \mathbb{N}_0,\,j<k,\, j+k \text{ is even}} (2j+n-3) C^\lambda_j(z)
\end{align*}
As another operator to be expressed in the Gegenbauer polynomial basis, let us consider $z^2\:C^\lambda_k(z)$. We want to express the operator $z^2$ in our Gegenbauer basis by using~\cite[(19)]{weisstein_gegenbauer}:
\begin{align*}
    z\:C^\lambda_k(z) = \frac{1}{2(k+\lambda)} \left((k+1)C_{k+1}^\lambda(z) +(k-1+2\lambda )C_{k-1}^\lambda(z) \right)\,.
\end{align*}
By using the recurrence relation twice, we obtain:
\begin{align*}
    z^2\:C^\lambda_k(z) = \eta_k^2\: C_{k+2}^\lambda(z) + \eta_k^1  \:C^\lambda_k(z) + \eta_k^0\: C_{k-2}^\lambda(z)\,.
\end{align*}
The result then follows by collecting of all terms.
\end{proof}

\section{Numerical method and results}
\label{Numerical Method}

\begin{figure}[ht]
\begin{center}
\includegraphics[width=16cm]{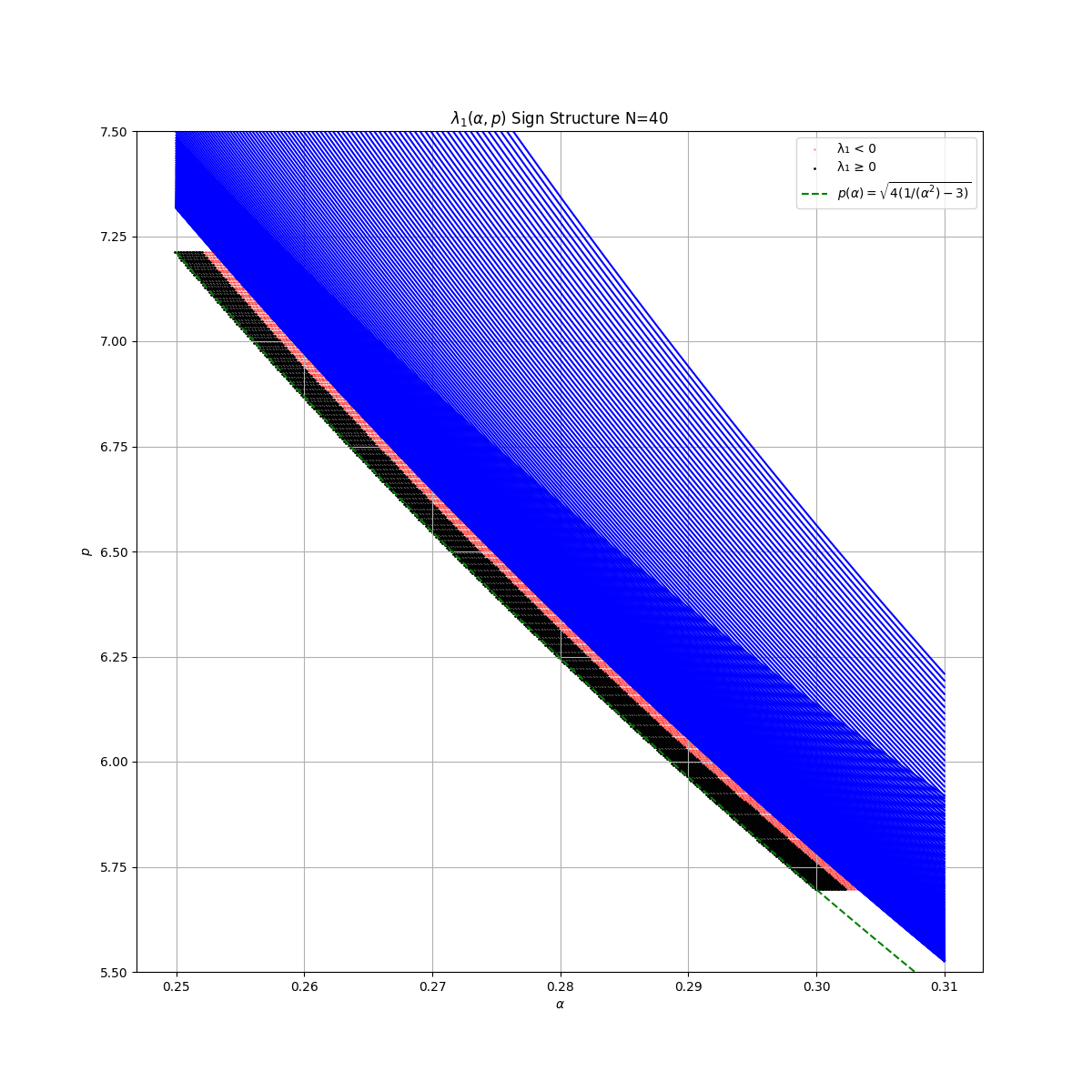}
\end{center}
\caption{\small Sign structure of the lowest eigenvalue to the truncated linearized problem in a part of the unknown region for $N=40$. The blue region and everything upwards of these lines is known to have non-symmetric minimizers. Everything lower than the green line is the region where \cite{bonheure_symmetry_2020} established symmetry of the minimizers.}
\label{fig:sign40stability_regions}
\end{figure}
\begin{figure}[ht]
\begin{center}
\includegraphics[width=16cm]{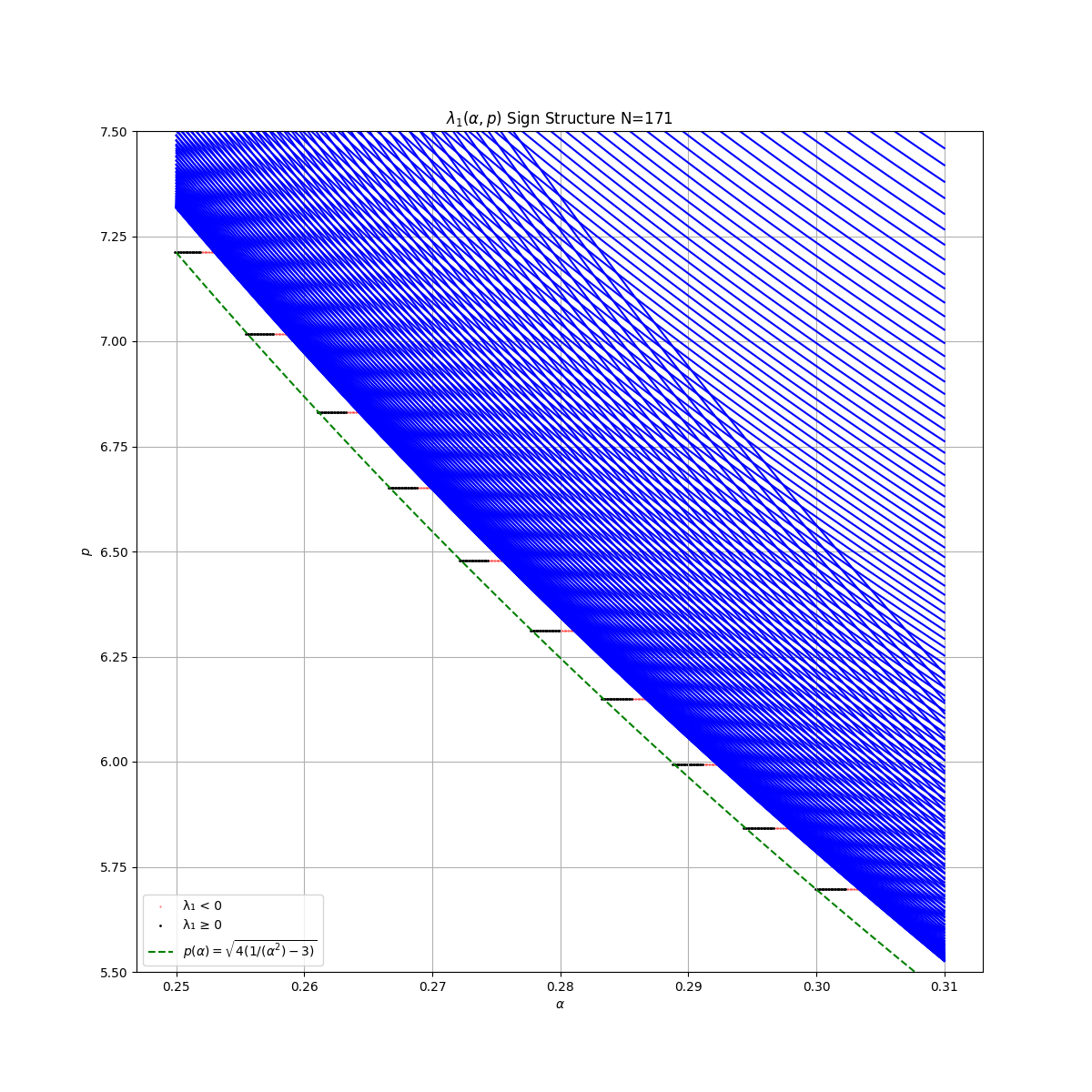}
\end{center}
\caption{\small Sign structure of the lowest eigenvalue to the truncated linearized problem in a part of the unknown region for $N=171$. The blue region and everything upwards of these lines is known to have non-symmetric minimizers. Everything lower than the green line is the region where \cite{bonheure_symmetry_2020} established symmetry of the minimizers. }
\label{fig:sign171stability_regions}
\end{figure}
The method used to do the numerics goes as follows.
As in Section~\ref{Reformulation in a Gegenbauer Polynomial Basis}, the operator $A$ is written as a $2\times 2$ block matrix $M$ with each block corresponding to an infinite dimensional matrix.
Each block is now truncated into its uppermost $N\times N$-block. This is a finite dimensional truncation of the whole infinite dimensional matrix. The finite dimensional truncation takes into account only the contributions from the first $N$ Gegenbauer polynomials. 
The lowest eigenvalue of this truncated matrix is then computed via standard sparse linear algebra solvers from the python SciPy Sparse Linear Algebra library.
We performed this computation for values $(\alpha,\,p)$ in the parameter space for which the $p$ is larger than the $p_*(\alpha)$ for which symmetry of the optimizing functions was established by~\cite{bonheure_symmetry_2020}.

The assumption needed for our numerics to obtain relevant results is, that the upper component (and respectively the lower component) of the lowest eigenfunction to $M$ can be well approximated by the first $N$ Gegenbauer polynomials.
Unsurprisingly, the numerically calculated lowest eigenfunction is radially symmetric and radially decreasing (see Figure \ref{fig:EigenvectorLinearizedestability_regions}). This can be analytically shown via a simple rearrangement inequality since $\phi_*$ is radially symmetric and radially decreasing. The numerically computed lowest eigenfunction puts mass on both the upper component and the lower component.

Furthermore, it seems to put mass only on the first ten or so even Gegenbauer polynomials, even when we put $N$ as high as $170$. We can take this as some numerical evidence, that our approximation is not significantly losing essential features. See Figures~\ref{fig:sign40stability_regions} and~\ref{fig:sign171stability_regions}.
\begin{figure}[ht]
\begin{center}
\includegraphics[width=10cm]{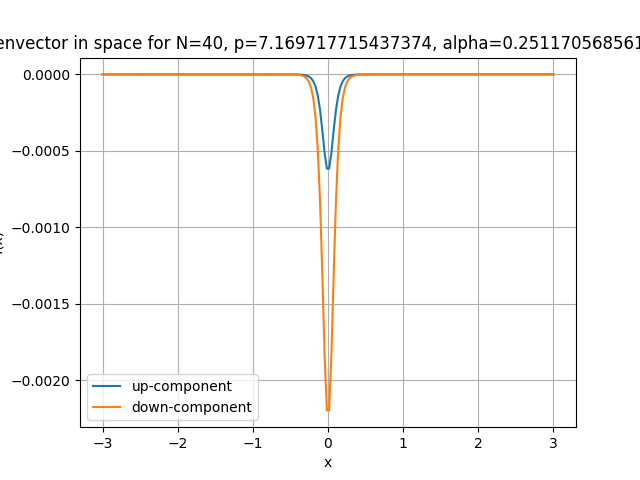}
\end{center}
\vspace*{-0.5cm}
\caption{\small Numerically computed eigenvector corresponding to the lowest eigenvalue for $A$ in $s\in \mathbb{R}$ coordinates for $N=40$, $p=7.169717715437374$, $\alpha=0.2511705685618729$. The radial decreasing symmetry is clearly visible.}
\label{fig:EigenvectorLinearizedestability_regions}
\end{figure}
\begin{figure}[hb]
\begin{center}
\includegraphics[width=12cm]{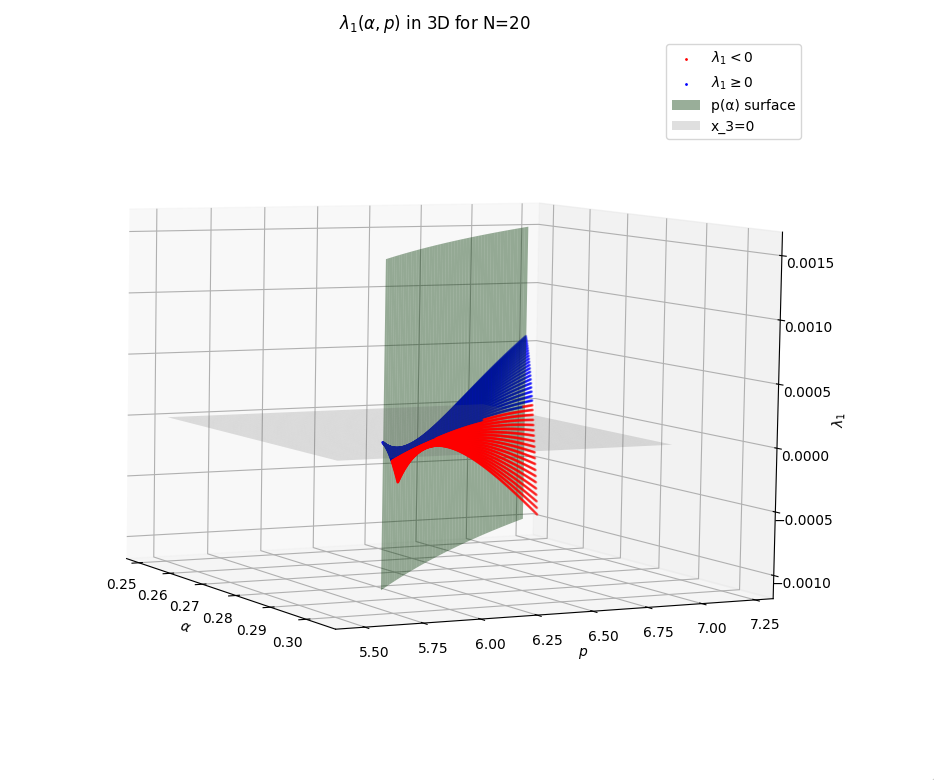}
\end{center}
\vspace*{-1.5cm}
\caption{\small Sign structure of the lowest eigenvalue to the truncated linearized problem in a part of the unknown region for $N=20$. Everything to the left of the green surface is the region where \cite{bonheure_symmetry_2020} established symmetry of the minimizers. The computations for the lowest $p$ for a given $\alpha$ lie on that surface. The figure indicates a second order phase transition at the threshold between linear stability and instability.}
\label{fig:3dstability_regions}
\end{figure}
\clearpage
\medskip\noindent{\bf Acknowledgment:} The authors wish to thank Maria J.~Esteban and Michael Loss for inspiring discussions, as the current work is an extension of~\cite{dolbeault_ckn_2025} and our discussions with them certainly influenced our way of thinking about the problem. Partial support through US National Science Foundation grant DMS-1954995 (R.L.F.), as well as through the German Research Foundation through EXC-2111-390814868 and TRR 352-Project-ID 470903074 (R.L.F.), and by the CONVIVIALITY Project (ANR-23-CE40-0003) of the French National Research Agency (J.D.) is acknowledged.\\
\noindent{\scriptsize\copyright\,\the\year\ by the authors. This paper may be reproduced, in its entirety, for non-commercial purposes. \href{https://creativecommons.org/licenses/by/4.0/legalcode}{CC-BY 4.0}}

\end{document}